\documentclass[12pt,a4paper]{amsart}
\usepackage{amsmath,amsfonts,amssymb}
\usepackage{setspace}

\usepackage{longtable}
\usepackage[mathscr]{eucal}
\usepackage{amsmath, amsthm}
\usepackage{mathrsfs}
\usepackage{amsbsy}
\usepackage{wasysym}
\usepackage{url}

\theoremstyle{plain}
\usepackage{color}
\allowdisplaybreaks[4]

\newtheorem{theorem}[subsection]{Theorem}

\newtheorem{lemma}[subsection]{Lemma}

\newtheorem{remark}[subsection]{Remark}

\newtheorem{corollary}[subsection]{Corollary}

\newtheorem{fact}[subsection]{Fact}

\input xypic
\xyoption{all}
\usepackage[breaklinks]{hyperref}
\theoremstyle{thmrm}

\usepackage{graphicx}
\begin{document}
\title{Diophantine Equation $x_1^{3}-x_2^{2}x_1+1=0$ over number fields}
\author {Pinki Khatun}

\address{SRM University AP India}
\email{pinki\_khatun@srmap.edu.in}
\thanks{During this work, author is engaged as a research fellow at SRM University AP. Author is grateful to supervisor Dr. Kalyan Banerjee to introduce with this family of elliptic curve.The author also wants to thank her co-supervisor Dr. Kalyan Chakraborty }
\maketitle
\begin{abstract}
	Let $\mathbb{K}=\mathbb{Q}(\sqrt{-d})\text{ or }\mathbb{Q}(\sqrt{d})$, where $d$ is a positive square-free integer, and denote by $\mathcal{O}_\mathbb{K}$ the ring of integers of $\mathbb{K}$. I investigate the solution of the equation
	$x_1^{3}-x_2^{2}x_1+1=0$ where $x_1\in \mathbb{K}$ and $x_2\in\mathcal{O}_\mathbb{K}$. The case $\mathbb{K}=\mathbb{Q}(\sqrt{d})$ for $d\equiv 2,3\pmod{4}$ faces infinite units that require a separate treatment. Using the arithmetic of the quadratic integer rings $\mathbb{Z}[\sqrt{d}]]$, together with norm arguments, divisibility properties, and the explicit structure of its unit group, I prove that the equation has exactly two solutions, namely
	$(x_1,x_2)=(-1,0)~\text{ and }~(1,\sqrt{2})$ for $\mathbb{K}=\mathbb{Q}(\sqrt{2})$ and one solution $(-1,0)$ for $\mathbb{K}=\mathbb{Q}(\sqrt{d})$
	As an application, I consider the family of elliptic curves $C_m:Y^{2}=X^{3}-m^{2}X+1,~
	m\in\mathcal{O}_\mathbb{K},$ and deduce that, for every $m\neq0,\sqrt{2}$ the Mordell--Weil group $C_m(\mathbb{K})$ contains no rational point of order two.
\end{abstract}


\tableofcontents
\section{Introduction}
Let $\mathbb{K}$ be one of the number fields $\mathbb{Q}(\sqrt{-d})$ or   $\mathbb{Q}(\sqrt{d})$ for some square free positive integer $d$
and let $\mathcal{O}_{\mathbb{K}}$ denote the ring of integers of
$\mathbb{K}$.

A polynomial equation $f(x_1,x_2,\ldots,x_k)=0$ is called a \emph{Diophantine equation over the number field $\mathbb{K}$} if
\[
f(x_1,x_2,\ldots,x_k)
=
\sum_{i_1,\ldots,i_k\geq 0}
A_{i_1,\ldots,i_k}
x_1^{i_1}\cdots x_k^{i_k},
\]
where $A_{i_1,\ldots,i_k}\in\mathcal{O}_\mathbb{K}$, and the variables
$x_1,\ldots,x_k$ are required to lie in $\mathcal{O}_\mathbb{K}$.
Equivalently, $f\in\mathcal{O}_\mathbb{K}[x_1,\ldots,x_k]
\quad\text{and}\quad
(x_1,\ldots,x_k)\in{\mathcal{O}_\mathbb{K}}^k.$

The study of Diophantine equations over number fields and their rings of integers is a classical topic in algebraic number theory. While many Diophantine equations have been extensively investigated over the rational integers, considerably less is known about their solvability over quadratic integer rings. The arithmetic of these rings, particularly the structure of their unit groups and norm forms, often provides powerful tools for determining the existence or non-existence of solutions.

In this article, I investigate the solution of the equation
\[
x_1^{3}-x_2^{2}x_1+1=0,
\]
where $x_1\in\mathbb{K} \quad \text{and} \quad
x_2\in\mathcal{O}_\mathbb{K}$.

Our primary objective is to determine whether this equation admits solutions when the parameter $x_2$ ranges over the ring of integers of the number field $\mathbb{K}$.

The arithmetic of $\mathbb{Z}[\sqrt{d}]$ is closely governed by its unit group, which is generated by the minimum unit $y$ greater than $1$. Exploiting this structure, together with norm arguments and divisibility properties in the principal ideal domain $\mathbb{Z}[\sqrt{d}]$, I derive several auxiliary Diophantine results concerning powers of the fundamental unit. These results enable us to eliminate all possible solutions of the above cubic equation.

In this short note, I prove that
\begin{theorem}
	\label{theorem1}
	Let $\mathbb{K}$ be one of the number fields $\mathbb{Q}(\sqrt{-d})$ or  $\mathbb{Q}(\sqrt{d})$ for some square free positive integer $d$
	and let $\mathcal{O}_{\mathbb{K}}$ denote the ring of integers of
	$\mathbb{K}$. Then the set $S$ of solutions
	$(x_1,x_2)\in\mathbb{K}\times\mathcal{O}_{\mathbb{K}}$ to the equation
	\[
	x_1^{3}-x_2^{2}x_1+1=0
	\]
	is given below,
	\begin{align*}
		S&=\{(-1,0),(1,\sqrt{2})\}&\text{ if }& \mathbb{K}=\mathbb{Q}(\sqrt{2}),&\\
		S&=\{(-1,0)\}&\text{ if }& \mathbb{K}=\mathbb{Q}(\sqrt{d}),&\text{ with condition C}\\
		S&=\{(-1,0)\}&\text{ if }& \mathbb{K}=\mathbb{Q}(\sqrt{-d})&d\neq 3\text{ and }\\
		S&=\{(-1,0),(-\omega,0),(-\omega^2,0)\}&\text{ if }& \mathbb{K}=\mathbb{Q}(\sqrt{-3})&
	\end{align*} 
	where $\omega=\dfrac{1+\sqrt{-3}}{2}$ and \begin{align}
		\text{condition C }:&``\text{ Suppose} y=\alpha+\beta\sqrt{d} \text{ is the } \text{ minimum unit greater than } 1\nonumber\\&\text{ in }\mathbb{Z}[\sqrt{d}] \text{ where } \alpha, \beta\text{ odd for } d\equiv2\pmod{4} \text{ and } \alpha \text{ even}\label{QuoteCond}\\&\text{ and } \beta \text{ odd for } d\equiv3\pmod{4}" \nonumber
	\end{align}.
\end{theorem}

As an application of this Diophantine result, I consider the family of elliptic curves
\[
C_m:\quad Y^{2}=X^{3}-m^{2}X+1,
\qquad
m\in\mathcal{O}_{\mathbb{K}}.
\]
Since the points of order two on $C_m$ correspond precisely to the roots of the cubic polynomial $X^{3}-m^{2}X+1,$ our main theorem immediately implies that, for every $m\in\mathcal{O}_{\mathbb{K}}\setminus\{0,\sqrt{2}\}$ the Mordell--Weil group $C_m(\mathbb{K})$ contains no point of order two. Hence the Diophantine theorem provides a direct arithmetic criterion for the absence of two-torsion points in this family of elliptic curves.

As a consequences of Theorem \ref{theorem1}, I have 
\begin{corollary}
	Let us consider a family of elliptic curves $ C_m : Y^2=X^3-m^2 X +1, \;m\in\mathcal{O}_{\mathbb{K}}$. For $ m\neq 0,\sqrt{2}$, there is no point of order $ 2 $ in the Mordel Weil group
	$C_m\big(\mathbb{K}\big)$ 
\end{corollary}
To know more on Elliptic curve and their torsion points please look at \cite{BruinNajman2016}, \cite{Kamienny1992}, \cite{KenkuMomose1988}, \cite{SilvermanTate2015}.
\begin{remark}
	\label{remark1}
	For $m=0$, the point $(-1,0)$ is the unique point of order $2$ in each of the groups
	$C_0\big(\mathbb{Q}(\sqrt{d})\big)$( with condition C \eqref{QuoteCond}) and $C_0\big(\mathbb{Q}(\sqrt{-d})\big)$  for square free positive integer $d\neq 3$. Also the point $(-1,0)$ is the unique point of order $2$ in the groups
	$C_0\big(\mathbb{Q}(\sqrt{3})\big)$.  On the other hand, the group $C_0\big(\mathbb{Q}(\sqrt{-3})\big)$ has three points of order $2$, namely $(-1,0)$, $(-\omega,0)$ and $(-\omega^2,0)$.
	For $m=\sqrt{2}$, the point	$(1,0)$ is the only point of order $2$ in the group
	$C_{\sqrt{2}}\big(\mathbb{Q}(\sqrt{2})\big)$.
	
\end{remark}

The techniques developed in this paper are entirely elementary and rely only on the arithmetic of the quadratic integer rings $\mathbb{Z}[\sqrt{d}]$. I hope that these methods will be useful in studying similar families of polynomial equations over other real quadratic fields.

The sections are organized as follows. In section \ref{sec:NumberField}, I studied the case $\mathbb{K}=\mathbb{Q}(\sqrt{-d})$. The case $\mathbb{K}=\mathbb{Q}(\sqrt{d})$ for $d\equiv2\pmod{4}$ is studied in section \ref{sec:mainsection-d2}. Also the case $\mathbb{K}=\mathbb{Q}(\sqrt{d})$ for $d\equiv3\pmod{4}$ is studied in section \ref{sec:mainsection3}.

\section{Proof of Theorem \ref{theorem1} for $\mathbb{K}=\mathbb{Q}(\sqrt{-d})$ }\label{sec:NumberField}
\begin{fact}\label{fact}
	Suppose $\alpha$ satisfy a polynomial $f(x)=x^n+a_{n-1}x^{n-1}+\cdots+a_0$ where $a_i$'s are algebraic integers. Then $\alpha$ is also an algebraic integer.
\end{fact}
\begin{proof}
	One can see that $\alpha$ is an algebraic number. Let $a_i$ satisfy a monic polinomial of degree $d_i$ with integer coefficients for $1\leq i\leq n-1$. Therefore $R=\mathbb{Z}[a_0,a_1,\cdots,a_{n-1}]$ and $R_1=\mathbb{Z}[a_0,a_1,\cdots,a_{n-1},\alpha ]$ can be written as a $\mathbb{Z}$-linear combination of set $S$ and $S_1$ respectively where
\begin{align*}
		S&=\{a_0^{k_0}a_1^{k_1}\cdots a_{n-1}^{k_{n-1}}~:~0\leq k_i\leq d_i,  0\leq i\leq n-1\}\text{ and }\\
		S_1&=\{a_0^{k_0}a_1^{k_1}\cdots a_{n-1}^{k_{n-1}}\alpha^{k_n}~:~0\leq k_i\leq d_i, 0\leq k_n\leq n, 0\leq i\leq n-1\}.
	\end{align*}
Thus $R$ and $R_1$ are finitely generated $\mathbb{Z}$-module with $\alpha R_1\subseteq R_1$. Viewing $(\mathbb{C},+)$ as a $\mathbb{Z}$ module it has no torsion element. As a $\mathbb{Z}$-submodule, $R_1$ has no torsion element. So $R_1$ is a torsion free finitely generated $\mathbb{Z}$ module. So it has a basis. Let $B=\{b_1,b_2,\cdots b_s\}$ be basis of the $\mathbb{Z}$-module $R_1$.

Also let $\alpha b_j=\sum_{i=1}^s m_{i,j}b_i$ for all $1\leq j\leq s$.

Let $V=span_\mathbb{Q} B$. Then $B$ is also a $\mathbb{Q}$-basis of $V$. Also consider the operator $T:V\to V$ such that $T(b_j)= \alpha b_j-[\sum_{i=1}^s m_{i,j}b_i]$ for all $j$. Then $T$ is $O$ operator and determinant of matrix representation of $T$ is $0$( because $T$ is not injective). 

Then consider the matrix $M$ whose $(i,j)$-th element is $m_{i,j}$. Consider the polynomial $f(x)=det(xI_s-M)\in M_s(\mathbb{Z})$. Then determinant of matrix representation of $T$ is $det(\alpha I_s-M)=f(\alpha)=0$. So $\alpha $ is algebraic integer.
\end{proof}
\begin{lemma}
	\label{lemma-unit}
	Let $(u,v)\in\mathbb{K}\times \mathcal{O}_\mathbb{K}$ is a solution of the equation
	\begin{align}
		x_1^{3}-x_2^{2}x_1+1=0,\label{equation-main}
	\end{align}
	then $u$ is unit in $\mathcal{O}_\mathbb{K}$. Here $\mathcal{O}_\mathbb{K}$ is ring of integers of the number field $\mathbb{K}$.
\end{lemma}
\begin{proof}
	Let $(u,v)\in\mathbb{K}\times \mathcal{O}_\mathbb{K}$ is solution of equation \eqref{equation-main}.
	
	Then I have $u^3-v^2u+1=0$ where $v\in \mathcal{O}_\mathbb{K}$.
	Then from the above fact \ref{fact} $u$ is an algebraic integer. So I have $u\in\mathcal{O}_\mathbb{K}$.
	Now I have $u(v^2-u^2)=1$. Therefore $u$ must be a unit in $\mathcal{O}_\mathbb{K}$.
\end{proof}

In the next table I enlist some number fields where number of units of ring of integers is a finite set.

$$
\begin{array}{|c|c|c|c|}
	\hline 
	\mathbb{K} && \mathcal{O}_{\mathbb{K}} & \mathcal{O}_{\mathbb{K}}^\times \\ \hline
	\mathbb{Q}(i)
	&&
	\mathbb{Z}[i]
	&
	\{\pm 1,\pm i\}
	\\\hline 
	\mathbb{Q}(\sqrt{-3})
	&\omega=\dfrac{1+\sqrt{-3}}{2}&
	\mathbb{Z}\left[\omega\right]
	&
	\left\{\pm 1,\ \pm\omega,\ \pm\omega^2\right\}
	
	\\[2mm]\hline 
	\mathbb{Q}(\sqrt{-d})
	&d\cong 2,1\pmod4&
	\mathbb{Z}\left[\sqrt{-d}\right]
	&
	\{\pm 1\}
	\\d\geq 5&&&\\[2mm]\hline 
	\mathbb{Q}(\sqrt{-d})
	&\beta_d=\dfrac{1+\sqrt{-d}}{2}&
	\mathbb{Z}\left[\beta_d\right]
	
	&
	\{\pm 1\}\\d\geq 5&d\cong 3\pmod4&&\\\hline
\end{array}
$$
 Therefore searching units $u$ such that $y^2=u^2+\frac{1}{u}$ for some $y\in\mathcal{O}_{\mathbb{K}}$, I have the following theorem.

\begin{theorem}[Theorem \ref{theorem1}]
	Diophantine Equation $x^3-y^2x+1=0$ has only one solution  $(-1,0)$ in ${\mathcal{O}_\mathbb{K}}^2$ for $\mathcal{O}_\mathbb{K}= \mathbb{Z}\left[\beta_d\right],\mathbb{Z}\left[\sqrt{-d}\right],\mathbb{Z}[i]$ and has solution set $\{(-1,0),(-\omega,0),(-\omega^2,0)\}$ in ${\mathcal{O}_\mathbb{K}}^2$ for $\mathcal{O}_\mathbb{K}= \mathbb{Z}\left[\omega\right]$
\end{theorem}

\section{Proof of Theorem \ref{theorem1} for $\mathbb{K}=\mathbb{Q}[\sqrt{d}]$ with $d\cong 2\pmod{4}$}\label{sec:mainsection-d2}
Suppose $d\cong 2\pmod{4}$. For the real quadratic number field $\mathbb{K}=\mathbb{Q}[\sqrt{d}]$ I have $\mathcal{O}_\mathbb{K}=\mathbb{Z}[\sqrt{d}]$. So the set all units in $\mathcal{O}_\mathbb{K}$ is  $\{\pm y^n~:~n\in\mathbb{Z}\}$, where $y$ is the minimum unit greater than $1$. Let $y=\alpha+\beta \sqrt{d}$ for some integer $\alpha,\beta$. Here I know that $x$ is unit if and only if $N(x)=\pm1$. Now I have the following
\[
N(y)\pmod{4}=\alpha^2-d\beta^2=\left\{\begin{matrix}
	1-2=-1&\text{ if }&\alpha \text{ is odd and }\beta \text{ is odd},\\
	1-0=~1&\text{ if }&\alpha \text{ is odd and }\beta \text{ is even},\\
	0-2=~2&\text{ if }&\alpha \text{ is even and }\beta \text{ is odd},\\
	0-0=~0&\text{ if }&\alpha \text{ is even and }\beta \text{ is even},\\
\end{matrix}\right.
\]
But $N(y)$ is allowed to be $\pm1$. Therefore I have following two cases:
\begin{enumerate}
	\item Suppose $\alpha$ is odd and $\beta$  is odd. Here I have $N(y)=-1$. Conjugate of $y$ is $\overline{y}=\alpha-\beta\sqrt{d}=-\frac{1}{y}$. As an example, $y=1+\sqrt{2}$ is fundamental unit in $\mathbb{Z}[\sqrt{2}]$.
	\item Suppose $\alpha$ is odd and $\beta$  is even. Here I have $N(y)=1$. Conjugate of $y$ is $\overline{y}=\alpha-\beta\sqrt{d}=\frac{1}{y}$. As an example, $y=5+2\sqrt{6}$ is fundamental unit in $\mathbb{Z}[\sqrt{6}]$.
\end{enumerate}
Throughout the rest part of this section suppose $y=\alpha+\beta\sqrt{d}$ for odd $\alpha,\beta$.

I first establish some auxiliary lemmas. These results will be used in the proof of the main theorem, which follows thereafter. 
\begin{lemma}
	\label{solution-equation-case1d2}
	The following equation 
	\[
	y^{4s}+y^{-2s}=m^2
	\]
	where positive $m(\neq \sqrt{2})\in\mathbb{Z}[\sqrt{d}]$ has no solution for natural number $s$.
\end{lemma}
\begin{proof}
	Here $y=\alpha+\beta\sqrt{d}>1$. If possible let there is a solution for positive integer $s$. I have the given equation
	\[
	y^{4s}+y^{-2s}=m^2\,\text{ for positive }m\in\mathbb{Q}[\sqrt{2}].
	\]
	Therefore I have $m>y^{2s}$. Calculating norm I get 
	\begin{align*}
		N(m+y^{2s})N(m-y^{2s})&=N(m^2-y^{4s})\\
		&=N(y^{-2s})=\pm1
	\end{align*}
	So $N(m+y^{2s})=\pm1$. Therefore $m+y^{2s}$ is unit in $\mathbb{Z}[\sqrt{d}]$. Here $m+y^{2s}>y>1$. So I have $m+y^{2s}=y^k$ for some natural number $k$. Therefore $m-y^{2s}=y^{-(2s+k)}$. Now I have the expression of $y^{3s}$ as follows;
	\begin{align*}
		y^{3s}&=y^{s}y^{2s}=\frac{y^{s}}{2}[(m+y^{2s})-(m-y^{2s})]\\
		&=\frac{1}{2} [y^{s+k}-y^{-(s+k)}]\\
		&=\frac{1}{2} [(\beta\sqrt{d}+\alpha)^{s+k}-(\beta\sqrt{d}-\alpha)^{s+k}]\\
		&=\frac{1}{2}\Big[ \sum_{r=0}^{s+k} \binom{s+k}{r}(\beta\sqrt{d})^r\big(\alpha^{s+k-r}-(-\alpha)^{s+k-r}\big)\Big]\\
		&=\left\{ \begin{matrix}
			\sum_{r~odd}^{s+k} \binom{s+k}{r}(\beta\sqrt{d})^r\alpha^{s+k-r}&\text{ if }&r\text{ is odd and }s+k \text{ is even}\\
			\sum_{r~even}^{s+k} \binom{s+k}{r}(\beta\sqrt{d})^r\alpha^{s+k-r}&\text{ if }&r\text{ is even and }s+k \text{ is odd}
		\end{matrix}   \right.\\~\\
		&=\left\{
		\begin{matrix}
			\displaystyle \left[\sum^{s+k}_{\substack{r=0\\r\ \text{even}}}
			\binom{s+k}{r}\beta^r\alpha^{s+k-r}d^{(\frac{r}{2}-1)}\right] +0\sqrt{d}
			& \text{if } s+k \text{ is odd}\\[1.2em]
			\displaystyle 0+\left[\sum^{s+k}_{\substack{r=0\\r\ \text{odd}}}
			\binom{s+k}{r}\beta^r\alpha^{s+k-r}d^{(\frac{r-1}{2}-1)}\right]\sqrt{d}
			& \text{if } s+k \text{ is even}
		\end{matrix}
		\right.
	\end{align*}
	Therefore I got a contradiction as $(\alpha+\beta\sqrt{d})^{3s}=a+b\sqrt{2}$ implies $a\neq 0\neq b$.
\end{proof}

\begin{lemma}
	\label{solution-equation-case2d2}
	The following equation 
	\[
	(\alpha-\beta\sqrt{d})^{4s}-(\alpha-\beta\sqrt{d})^{-2s}=m^2
	\]
	where positive $m(\neq \sqrt{2})\in\mathbb{Z}[\sqrt{d}]$ has no solution for natural number $s$.
\end{lemma}
\begin{proof}
	As $-1<\alpha-\beta\sqrt{d}=\frac{-1}{\alpha+\beta\sqrt{d}}<0$, I have $0<(\alpha-\beta\sqrt{d})^{4s}<(\alpha-\beta\sqrt{d})^{2s}<1$. So $	(\alpha-\beta\sqrt{d})^{4s}-(\alpha-\beta\sqrt{d})^{-2s}$ is negative real number but $m^2$ is positive number. 
\end{proof}

Now I am ready to proof the Theorem \ref{theorem1}. I restate Theorem~\ref{theorem1} for convenience and then process the proof.

\begin{lemma}\label{unit-root2-coefd2}
	If $(\alpha\pm\beta\sqrt{d})^{n}=a^{\pm}_n+b^{\pm}_n\sqrt{d},$ then $b^{\pm}_n$ is odd or even if and only if $n$ is so.
\end{lemma}
\begin{proof}
	I always have $m\cong -m~(mod~2)$. Here I can see that
	\begin{align*}
		b_n^{\pm}\sqrt{d}&=(\pm1)^n\Big[\sum_{\substack{r=0\\r~odd}}^{n}\binom{n}{r}(\beta \sqrt{d})^r(\pm \alpha )^{n-r}\Big]\\
		&=(\pm1)^n\Big[n(\pm \alpha)^{n-1}\beta+\sum_{\substack{r=3\\r~odd}}^{n}\binom{n}{r}\beta^rd^{(\frac{r-1}{2})}(\pm \alpha)^{n-r}\Big]\sqrt{d}.
	\end{align*}
	So I have \[b_n^{\pm}\cong (\pm1)^n\Big[n(\pm \alpha)^{n-1}\beta+\sum_{\substack{r=3\\r~odd}}^{n}\binom{n}{r}\beta^rd^{(\frac{r-1}{2})}(\pm \alpha)^{n-r}\Big]\cong n~(mod ~2).\]
\end{proof}

\begin{theorem}[Theorem \ref{theorem1}]
	\label{theorem2d2}
	The solution set $S$ of the equation
	\[
	x_1^{3}-x_2^{2}x_1+1=0
	\]
	in $\mathbb{Q}(\sqrt{d})\times \mathbb{Z}[\sqrt{d}]$ is $\{(-1,0),(1,\sqrt{2})\}$ for $d=2$ and $\{(-1,0)\}$ for $d\neq 2$ .	
\end{theorem}
\begin{proof}
	One can check that $(-1,0)$ is a solution of the equation and also $(1,\sqrt{2})$ is a solution for $d=2$. Let $(x,m)$ be any other solution of the given equation. Then $x\neq 0$ and $m\neq 0,\sqrt{2}$.
	
	From the Lemma \ref{lemma-unit}, I have  $x$ is a unit in $\mathbb{Z}[\sqrt{d}]$.
	
	Here
	\[
	y=\alpha+\beta\sqrt{d} \quad \text{and suppose} \quad z=\alpha-\beta\sqrt{d}.
	\]
	I know that all units of $\mathbb{Z}[\sqrt{d}]$ are of the form (see \ref{units-theory2})
	\[
	\pm y^n,\ \pm z^n,
	\]
	where $n$ is a non-negative integer.
	
	As $m\neq 0,\sqrt{2}$, I must have $n>0$.
	Therefore, I have the following four possibilities:
	\begin{align}
		y^{2n}+y^{-n} &= m_1^2, \label{point2case1d2}\\
		z^{2n}-z^{-n} &= m_2^2, \label{point2case2d2}\\
		y^{2n}-y^{-n} &= m_3^2, \label{point2case3d2}\\
		z^{2n}+z^{-n} &= m_4^2. \label{point2case4d2}
	\end{align}
	Assuming  $y^k=a^+_k+b^{+}_k\sqrt{2} $ and $z^k=a^-_k+b^{-}_k\sqrt{2} $ for integer $k$, and comparing the coefficients of $\sqrt{d}$ in the above four cases I have
	\begin{align*}
		b^{+}_{2n}\pm (-1)^nb^{+}_{n}&\cong 0~(mod~2)\text{ and }\\
		b^{-}_{2n}\pm (-1)^nb^{-}_{n}&\cong 0~(mod~2)
	\end{align*} But from lemma \ref{unit-root2-coefd2}, I have $b^{\pm}_{2n}$ is even and $ b^{\pm}_{n}\cong n (mod~2)$. So I have $n\cong 0(mod~2)$, that is $n$ is even.
	
	Taking conjugates, equation~\eqref{point2case1d2} and equation~\eqref{point2case2d2} become equation~\eqref{point2case4d2} and equation~\eqref{point2case3d2}, respectively. Therefore  equation~\eqref{point2case1d2} and equation~\eqref{point2case2d2} have solution if and only if  equation~\eqref{point2case4d2} and equation~\eqref{point2case3d2}, respectively have solution.
	Now from lemma \ref{solution-equation-case1d2} and lemma \ref{solution-equation-case2d2}  equation~\eqref{point2case1d2} and equation~\eqref{point2case2d2} have no solution.Therefore the above four possibilities have no solutions.

\end{proof}

\section{Proof of Theorem \ref{theorem1} for $\mathbb{K}=\mathbb{Q}[\sqrt{d}]$ with $d\cong 3\pmod{4}$}\label{sec:mainsection3}
Suppose $d\cong 3\pmod{4}$. For the real quadratic number field $\mathbb{K}=\mathbb{Q}[\sqrt{d}]$ I have $\mathcal{O}_\mathbb{K}=\mathbb{Z}[\sqrt{d}]$. So the set all units in $\mathcal{O}_\mathbb{K}$ is  $\{\pm y^n~:~n\in\mathbb{Z}\}$, where $y$ is the minimum unit greater than $1$. Let $y=\alpha+\beta \sqrt{d}$ for some integer $\alpha,\beta$. Here I know that $x$ is unit if and only if $N(x)=\pm1$. Now I have the following
\[
N(y)\pmod{4}=\alpha^2-d\beta^2=\left\{\begin{matrix}
	1-3=2&\text{ if }&\alpha \text{ is odd and }\beta \text{ is odd},\\
	1-0=~1&\text{ if }&\alpha \text{ is odd and }\beta \text{ is even},\\
	0-3=~1&\text{ if }&\alpha \text{ is even and }\beta \text{ is odd},\\
	0-0=~0&\text{ if }&\alpha \text{ is even and }\beta \text{ is even},\\
\end{matrix}\right.
\]
But $N(y)$ is allowed to be $\pm1$. Therefore I have following two cases:
\begin{enumerate}
	\item Suppose $\alpha$ is even and $\beta$  is odd.  As an example, $y=2+\sqrt{3}$ is fundamental unit in $\mathbb{Z}[\sqrt{3}]$.
	\item Suppose $\alpha$ is odd and $\beta$  is even. As an example, $y=25+4\sqrt{39}$ is fundamental unit in $\mathbb{Z}[\sqrt{39}]$.
\end{enumerate}
Here I have $N(y)=1$. Conjugate of $y$ is $\overline{y}=\alpha-\beta\sqrt{d}=\frac{1}{y}$.

Throughout the rest part of this section suppose $y=\alpha+\beta\sqrt{d}$ for odd $\beta$ and even $\alpha$.

I first establish some auxiliary lemmas.  These results will be used in the proof of the main theorem, which follows thereafter.
\begin{lemma}
	\label{solution-equation-case13}
	The following equation 
	\[
	(\alpha+\beta\sqrt{d})^{4s}+(\alpha+\beta\sqrt{d})^{-2s}=m^2
	\]
	where positive $m\in\mathbb{Z}[\sqrt{d}]$ has no solution for natural number $s$.
\end{lemma}
\begin{proof}
	Here $y=\alpha+\beta\sqrt{d}>1$.  If possible let there is a solution for positive integer $s$. I have the given equation
	\[
	y^{4s}+y^{-2s}=m^2\,\text{ for positive }m\in\mathbb{Z}[\sqrt{d}].
	\]
	Therefore I have $m>y^{2s}$. Calculating norm I get 
	\begin{align*}
		N(m+y^{2s})N(m-y^{2s})&=N(m^2-y^{4s})\\
		&=N(y^{-2s})=\pm1
	\end{align*}
	So $N(m+y^{2s})=\pm1$. Therefore $m+y^{2s}$ is unit in $\mathbb{Z}[\sqrt{d}]$. Here $m+y^{2s}>y>1$. So from Theorem \ref{units-theory3}, I got $m+y^{2s}=y^k$ for some natural number $k$. Therefore $m-y^{2s}=y^{-(2s+k)}$. Now I have the expression of $y^{3s}$ as follows;
	\begin{align*}
		y^{3s}&=y^{s}y^{2s}=\frac{y^{s}}{2}[(m+y^{2s})-(m-y^{2s})]\\
		&=\frac{1}{2} [y^{s+k}-y^{-(s+k)}]\\
		&=\frac{1}{2} [(\alpha+\beta\sqrt{d})^{s+k}-(\alpha-\beta\sqrt{d})^{s+k}]\\
		&=\frac{1}{2}\Big[ \sum_{r=0}^{s+k} \binom{s+k}{r}\alpha^r(\beta\sqrt{d})^{s+k-r}\big(1-(-1)^{s+k-r}\big)\Big]\\
		&=\frac{1}{2}\Big[ \sum_{t=0}^{s+k} \binom{s+k}{t}\alpha^{s+k-t}(\beta\sqrt{d})^t\big(1-(-1)^{t}\big)\Big]\\
		&=\left\{ \begin{matrix}
			\sum_{t~odd}^{s+k} \binom{s+k}{t}\alpha^{s+k-t}(\beta\sqrt{d})^t&\text{ if }&s+k \text{ is odd}\\~\\
			\sum_{t~odd}^{s+k-1} \binom{s+k}{r}\alpha^{s+k-t}(\beta\sqrt{d})^t&\text{ if }&s+k \text{ is even}
		\end{matrix}   \right.\\~\\
		&=\left\{ \begin{matrix}
		0+\Big[\sum_{r=0}^{\frac{1}{2}(s+k-1)} \binom{s+k}{2r+1}\alpha^{s+k-2r-1}\beta^{2r+1}d^r\Big]\sqrt{d}&\text{ if }&s+k \text{ is odd}\\~\\
		0+\Big[\sum_{r=0}^{\frac{1}{2}(s+k-2)} \binom{s+k}{2r+1}\alpha^{s+k-2r-1}\beta^{2r+1}d^r\Big]\sqrt{d}&\text{ if }&s+k \text{ is even}
		\end{matrix}   \right.\\~\\
	\end{align*}
	Therefore I got a contradiction as $(\alpha+\beta\sqrt{d})^{3s}=a+b\sqrt{d}$ implies $a\neq 0\neq b$.
\end{proof}

\begin{lemma}
	\label{solution-equation-case23}
	The following equation 
	\[
	(\alpha-\beta\sqrt{d})^{4s}-(\alpha-\beta\sqrt{d})^{-2s}=m^2
	\]
	where positive $m\in\mathbb{Z}[\sqrt{d}]$ has no solution for natural number $s$.
\end{lemma}
\begin{proof}
	As $0<\alpha-\beta\sqrt{d}<1$, I have $0<(\alpha-\beta\sqrt{d})^{4s}<(\alpha-\beta\sqrt{d})^{2s}<1$. So $	(\alpha-\beta\sqrt{d})^{4s}-(\alpha-\beta\sqrt{d})^{-2s}$ is negative real number but $m^2$ is positive number. 
\end{proof}

Now I am ready to proof the Theorem \ref{theorem1}. I restate Theorem~\ref{theorem1} for convenience and then process the proof.

\begin{lemma}\label{unit-root3-coef}
	If $(\alpha\pm \beta\sqrt{d})^{n}=a^{\pm}_n+b^{\pm}_n\sqrt{d},$ then $a^{\pm}_n$ is odd and $b^{\pm}_n$ is even for even $n$ and $a^{\pm}_n$ is even and $b^{\pm}_n$ is odd for odd $n$. Therefore $b^{\pm}_n\cong n\pmod{2}$.
\end{lemma}
Binomial expansion and comparing coefficients are enough to get the above statement.

\begin{theorem}[Theorem \ref{theorem1}]
	\label{theorem3}
	The solution set of the equation
	\[
	x_1^{3}-x_2^{2}x_1+1=0
	\]
	is $\{(-1,0)\}$ in $\mathbb{Q}(\sqrt{d})\times \mathbb{Z}[\sqrt{d}]$.	
\end{theorem}
\begin{proof}
	One can check that $(-1,0)$ satisfy the equation. Let $(x,m)$ be any other solution of the given equation. Then $x\neq 0$ and $m\neq 0$.
	
	From the Lemma \ref{lemma-unit}, I have  $x$ is a unit in $\mathbb{Z}[\sqrt{d}]$.
	
	Here I know
	\[
	y=\alpha+\beta\sqrt{d} \quad \text{and suppose} \quad z=\alpha-\beta\sqrt{d}.
	\]
	I know that all units of $\mathbb{Z}[\sqrt{3}]$ are of the form
	\[
	\pm y^n,\ \pm z^n,
	\]
	where $n$ is a non-negative integer. Here I can see that $y^{-1}=z$. For $n=0$, I obtain the units $\pm 1$, which imply that $m=0$.
	As $m\neq 0$, I must have $n>0$.
	Therefore, I have the following four possibilities:
	\begin{align}
		y^{2n}+y^{-n} &= m_1^2, \label{point2case13}\\
		z^{2n}-z^{-n} &= m_2^2, \label{point2case23}\\
		y^{2n}-y^{-n} &= m_3^2, \label{point2case33}\\
		z^{2n}+z^{-n} &= m_4^2. \label{point2case43}
	\end{align}
	Assuming  $y^k=a^+_k+b^{+}_k\sqrt{d} $ and $z^k=a^-_k+b^{-}_k\sqrt{d} $ for integer $k$, and comparing the coefficients of $\sqrt{d}$ in the above four cases I have
	\begin{align*}
		b^{+}_{2n}\pm b^{-}_{n}&\cong 0~(mod~2)\text{ and }\\
		b^{-}_{2n}\pm b^{+}_{n}&\cong 0~(mod~2)
	\end{align*} But from lemma \ref{unit-root3-coef}, I have $b^{\pm}_{2n}$ is even and $ b^{\pm}_{n}\cong n (mod~2)$. So I have $n\cong 0(mod~2)$, that is $n$ is even.
	
	Taking conjugates, equation~\eqref{point2case13} and equation~\eqref{point2case23} become equation~\eqref{point2case43} and equation~\eqref{point2case33}, respectively. Therefore  equation~\eqref{point2case13} and equation~\eqref{point2case23} have solution if and only if  equation~\eqref{point2case43} and equation~\eqref{point2case33}, respectively have solution.
	Now from lemma \ref{solution-equation-case13} and lemma \ref{solution-equation-case23}  equation~\eqref{point2case13} and equation~\eqref{point2case23} have no solution.Therefore the above four possibilities have no solutions.
\end{proof}

\section{Conclusion}

In this paper, I completely determine the solution set of the Diophantine equation
\[
x_1^{3}-x_2^{2}x_1+1=0
\]
over the number field $\mathbb{K}$.

As an application, I consider the family of elliptic curves
\[
C_m:Y^2=X^3-m^2X+1,
\qquad
m\in\mathcal{O}_\mathbb{K},
\]
and deduce that, for every $m\neq0,\sqrt{2}$, the Mordell--Weil group $C_m\bigl(\mathbb{K}\bigr)$ 
contains no rational point of order two. Although our proof is entirely elementary, this result complements the extensive study of torsion subgroups of elliptic curves over quadratic fields initiated by Kamienny \cite{Kamienny1992} and Kenku--Momose \cite{KenkuMomose1988}, and is consistent with the general philosophy that arithmetic properties of defining equations can provide useful information about the Mordell--Weil group. Related approaches to torsion over number fields can also be found in \cite{BruinNajman2016}, while general background on the arithmetic of elliptic curves is available in \cite{SilvermanTate2015}. Our family is also related to earlier investigations of parametrized families of elliptic curves, such as those considered in \cite{Antoniewicz2005}.

The methods developed in this paper suggest several natural directions for further investigation. One may study analogous Diophantine equations over the rings of integers of other real quadratic fields, replace the constant term by a general integer parameter, or consider more general families of cubic equations whose solvability influences the arithmetic of associated elliptic curves. I hope that the elementary techniques employed here will be useful in these broader settings.
Another is the study of torsion points of higher orders of the family $C_m$.

\appendix
\section{Some properties about  $\mathbb{Z}[\sqrt{2}]$ }
Let $d$ is squarefree positive integer greater than 1. Suppose $N:\mathbb{Q}[\sqrt{d}]\longrightarrow\mathbb{Q}$ is a function such that
\[
N(a+b\sqrt{2})=a^2-db^2.
\]
Then $N$ is norm on $\mathbb{Q}[\sqrt{d}]$ and $N$ is multiplicative and $N(\mathbb{Z}[\sqrt{d}])\subseteq\mathbb{Z}$.
Let $v:\mathbb{Z}[\sqrt{d}]\to\mathbb{Z}$ such that $v(a+b\sqrt{d})=\mid N(a+b\sqrt{d})\mid $.

Suppose $d=2$ or $3$ for the rest part.
\begin{theorem}
	{\bf\color{blue} (Remove it before publish)} The rings $\mathbb{Z}[\sqrt{2}]$ and $\mathbb{Z}[\sqrt{3}]$ are Euclidean
	domains with respect to the absolute value of the norm. 
\end{theorem}

\begin{proof}
	I first consider $R=\mathbb{Z}[\sqrt{d}]$, where $d=2$ or $3$.
	
	I want to show that $v$ is a Euclidean function.
	
	Let $\alpha,\beta\in R$ with $\beta\neq0$. Put
	\[
	\frac{\alpha}{\beta}=x+y\sqrt{d},
	\qquad x,y\in\mathbb{Q}.
	\]
	By the Archimedean property of $\mathbb{R}$, for every real number
	$x$ there exists $m\in\mathbb{Z}$ such that
	\[
	\left|x-m\right|\leq\frac12.
	\]
	Similarly, there exists $n\in\mathbb{Z}$ such that
	\[
	\left|y-n\right|\leq\frac12.
	\]
	Choose
	\[
	q=m+n\sqrt{d}\in R.
	\]
	Then
	\[
	\frac{\alpha}{\beta}-q
	=(x-m)+(y-n)\sqrt{d}.
	\]
	Put
	\[
	r=\alpha-q\beta.
	\]
	Clearly $r\in R$, and
	\[
	r=\beta\left(\frac{\alpha}{\beta}-q\right).
	\]
	Hence, by the multiplicativity of the norm,
	\[
	v(r)
	=
	v(\beta)
	\left|
	N\left((x-m)+(y-n)\sqrt{d}\right)
	\right|.
	\]
	Now
	\[
	\begin{aligned}
		\left|
		N\left((x-m)+(y-n)\sqrt{d}\right)
		\right|
		&=
		\left|(x-m)^2-d(y-n)^2\right|\\
		&\leq
		\max\left\{(x-m)^2,d(y-n)^2\right\}\\
		&\leq
		\frac{d}{4}.
	\end{aligned}
	\]
	For $d=2$ or $3$, I have
	\[
	\frac{d}{4}<1.
	\]
	Therefore
	\[
	v(r)
	\leq
	\frac{d}{4}v(\beta)
	<
	v(\beta).
	\]
	Since
	\[
	\alpha=q\beta+r,
	\]
	I have found $q,r\in R$ such that
	\[
	\alpha=q\beta+r
	\qquad\text{and}\qquad
	v(r)<v(\beta).
	\]
	Thus $v$ is a Euclidean function on $R$.
	
\end{proof}

\begin{lemma} {\bf\color{blue} (Remove it before publish)}
	The followings are equivalent:
	\begin{enumerate}
		\item[I.] $(a+b\sqrt{2})$ is unit in $\mathbb{Z}[\sqrt{2}]$.
		\item[II.] $(-a+b\sqrt{2})$ is unit in $\mathbb{Z}[\sqrt{2}]$.
		\item[III.] $(a-b\sqrt{2})$ is unit in $\mathbb{Z}[\sqrt{2}]$.
		\item[IV.] $(-a-b\sqrt{2})$ is unit in $\mathbb{Z}[\sqrt{2}]$.
		\item[V.] $N(a+b\sqrt{2})=\pm 1$.
	\end{enumerate}
\end{lemma}
\begin{proof}
	
	Let $x=a+b\sqrt{2}$. If $a=0$, then $N(b\sqrt{2})=-2b^2\neq \pm 1$  and $x$ is not a unit. If $b=0$, then $N(a)=a^2\neq \pm 1$  and $x$ is not a unit unless $a=\pm 1$. Now consider $a^2-2b^2=\pm 1$ with $a\neq 0\neq b$. Without loss of generality let $a,b>0$. Therefore $b\geq 1$. So $1\leq b^2<2b^2$. 
	
	Therefore $1\leq b^2 \implies b^2\leq 2b^2-1\leq 2b^2\pm 1=a^2 \implies b\leq a$ and
	
	$1\leq 2b^2\implies a^2=2b^2\pm 1\leq 2b^2+1< 4b^2 \implies a< 2b$
	
	So $a^2-2b^2=\pm 1 \text{ and } ab\neq 0\implies \mid b\mid \leq \mid a\mid < 2\mid b \mid$.
\end{proof}

\begin{lemma} {\bf\color{blue} (Remove it before publish)}
	Let $S=\{a+b\sqrt{2}~:~a,b\geq 0~\}\subseteq \mathbb{Z}[\sqrt{2}]$. I know that unit $a+b\sqrt{2}$ in $\mathbb{Z}[\sqrt{2}]$ with $0\neq b$ satisfy $\mid b\mid \leq \mid a\mid < 2\mid b \mid$. If $x \in S$ is invertible in $\mathbb{Z}[\sqrt{2}]$, then 
	$x = (1+\sqrt{2})^n$ for some positive integer $n$.
\end{lemma}	
\begin{proof}
	
	First notice that $(\sqrt{2}-1)(\sqrt{2}+1) = 1.$
	So $(1+\sqrt{2})$ is a unit.
	
	Now I use strong induction on $b$ for any unit $(a + b\sqrt{2}) \in S$.
	
	For $b=1$, the possible values of $a$ is $1$.  
	Here  $(1+\sqrt{2})$ is  invertible.  
	Therefore the statement is true for $b=1$.
	
	Assume the statement is true for $b = 1,2,\dots,k$.
	
	Now consider a unit $(a + b\sqrt{2})\in S$ where $b = k+1$.
	
	Then
	\[
	(a+(k+1)\sqrt{2})(\sqrt{2}-1)
	= (2k+2-a) + (a-k-1)\sqrt{2} \in S.
	\]
	
	as $(k+1) \le a < (2k+2)$. Also $ (2k+2-a)+(a-k-1)\sqrt{2}$ is unit in $\mathbb{Z}[\sqrt{2}]$ and $ a-k-1 \leq k.$
	
	So $(2k+2-a)+(a-k-1)\sqrt{2}=(1+\sqrt{2})^m $ for some natural number $m$, by the strong induction hypothesis.
	
	Multiplying both sides by $(1+\sqrt{2})$, I get
	\[
	a + (k+1)\sqrt{2}
	= \big[(2k+2-a) + (a-k-1)\sqrt{2}\big](1+\sqrt{2})
	= (1+\sqrt{2})^{m+1}.
	\]
	Hence all units in $S$ are of the given form. 
\end{proof}
	
	\begin{lemma} {\bf\color{blue} (Remove it before publish)}
		Let $T=\{a-b\sqrt{2}~:~a,b\geq 0\}\subseteq \mathbb{Z}[\sqrt{2}]$. I know that unit $a+b\sqrt{2}$ in $\mathbb{Z}[\sqrt{2}]$ with $0\neq b$ satisfy $\mid b\mid \leq \mid a\mid < 2\mid b \mid$.
		
		If $x \in T$ is invertible in $\mathbb{Z}[\sqrt{2}]$, then 
		\[
		x = (1-\sqrt{2})^n
		\]
		for some positive integer $n$.	
	\end{lemma}	
	\begin{proof}
		Now I use strong induction on $b$ for any unit $a - b\sqrt{2} \in T$.
		
		For $b=1$, the possible value of $a$ is $1$.  
		
		Therefore the statement is true for $b=1$.
		
		Assume the statement is true for $b = 1,2,\dots,k$,$\mid b\mid \geq 2$
		
		Now consider a unit $a-b\sqrt{2}$ where $b = k+1$.
		
		Then $(a-(k+1)\sqrt{2})(-1 - \sqrt{2})
		= (2k+2-a) - (a-k-1)\sqrt{2}.$
		
		Since $(k+1) \le a \le (2k+2)$, I have $a-k-1 \le k.$
		
		Also $(2k+2-a) - (a-k-1)\sqrt{2}$ is a unit.  
		By the strong induction hypothesis, $(2k+2-a) - (a-k-1)\sqrt{2} = (1-\sqrt{2})^m$
		for some natural number $m$.
		
		Multiplying both sides by $(1-\sqrt{2})$, I get
		\[
		a - (k+1)\sqrt{2}
		= \big[(2k+2-a) + (a-k-1)\sqrt{2}\big](1+\sqrt{2})
		= (1-\sqrt{2})^{m+1}.
		\]
		
		Hence all units in $T$ are of given form.
		\end{proof}
		\begin{theorem} {\bf\color{blue} (Remove it before publish)}
			The unit group is $U(\mathbb{Z}[\sqrt{2}]) = \{ \pm (1\pm \sqrt{2})^n : n \in \mathbb{N} \}\cup\{\pm 1\}$.
		\end{theorem}
\begin{proof}
	Suppose $x$ is unit in $\mathbb{Z}[\sqrt{2}]$. If $x$ is integer then the only possibility is $\pm 1$. No unit is integer multile of $\sqrt{2}$.
	
	Suppose $x=a+ b\sqrt{2}$ where $a,b$ are positive integer. Then $x$ is of the form $(1+\sqrt{2})^n$.
	
	Suppose $x=a- b\sqrt{2}$ where $a,b$ are positive integer. Then $x$ is of the form $(1-\sqrt{2})^n$.
	
	Suppose $x=-a- b\sqrt{2}=-(a+ b\sqrt{2})$ where $a,b$ are positive integer. Then $x$ is of the form $-(1+\sqrt{2})^n$.
	
	Suppose $x=-a+ b\sqrt{2}=-(a- b\sqrt{2})$ where $a,b$ are positive integer. Then $x$ is of the form $-(1-\sqrt{2})^n$.
\end{proof}

\begin{theorem}\label{units-theory2}
	Let $y=\alpha+\beta\sqrt{d}$ is the minimum number which is unit and greater than 1 in $\mathbb{Z}[\sqrt{d}]$ with $N(y)=-1$ and $d\equiv 2\pmod{4}$. Then the units of $\mathbb{Z}[\sqrt{d}]$ are given below:
	\begin{enumerate}
		\item units $y^n$ are positive and greater than $1$,
		\item units $-y^n$ are negative and less than $-1$,
		\item units $(\overline{y})^n$ is positive and lies between $0$ and $1$ for even $n$ and  units $-(\overline{y})^n$ is positive and lies between $0$ and $1$ for odd $n$,
		\item   units $(\overline{y})^n$ is negative and lies between $-1$ and $0$ for odd $n$ and  units $-(\overline{y})^n$ is negative and lies between $-1$ and $0$ for even $n$.
	\end{enumerate}
\end{theorem}

\begin{theorem}\label{units-theory3}
	Let $y=\alpha+\beta\sqrt{d}$ is the minimum number which is unit and greater than 1 in $\mathbb{Z}[\sqrt{d}]$ and $d\equiv 3\pmod{4}$.Then the units of $\mathbb{Z}[\sqrt{d}]$ are given below:
	\begin{enumerate}
		\item units $y^n$ are positive and greater than $1$,
		\item units $-y^n$ are negative and less than $-1$,
		\item units $(\overline{y})^n$ is positive and lies between $0$ and $1$ 
		\item   units $-(\overline{y})^n$ is negative and lies between $-1$ and $0$.
	\end{enumerate}
\end{theorem}

\begin{theorem}(Primes in {$\mathbb{Z}[\sqrt{2}]$}) {\bf\color{blue} (Remove it before publish)}
	Let $p$ is a positive odd prime integer. Then I have the following facts. 
	\begin{itemize}
	\item If $p\cong 3,5~(mod~8)$ then it is irreducible in $\mathbb{Z}[\sqrt{2}]$. I see that it's conjugate is itself. I denote it by $\alpha_p$.
	\item If $p\cong 1,7~(mod~8)$, there exist two irreducibles $u\pm v\sqrt{2}$ such that either $N(u\pm v\sqrt{2})=+p$ or $N(u\pm v\sqrt{2})=-p$ with odd $u$ and non zero $v$. Here I denote $u+ v\sqrt{2}$ by $\alpha_p$.
	\item I also have $\sqrt{2}$ is irreducible in $\mathbb{Z}[\sqrt{2}]$. I denote $\sqrt{2}$ by $\alpha_2$. 
	\end{itemize}
	Then $\{\alpha_p: p \text{ is prime}\}\cup \{\overline{\alpha_p}: p \text{ is prime and }p\cong 1,7\pmod{8} \}$ is the list of all irreducibles in $\mathbb{Z}[\sqrt{2}]$ upto associates.
\end{theorem}
\begin{fact} {\bf\color{blue} (Remove it before publish)}
	Let $p$ is odd prime. Then 2 is a square in $\mathbb{Z}_p$ iff $p\cong \pm 1~(mod~8)$.
\end{fact}
\begin{proof}
	As $p$ is odd prime, then either $p\cong\pm 1~(mod~8)$ or $p\cong\pm 3~(mod~8)$.
	
	Let $\alpha$ be a root of $x^4+1=0$ in some extension $\mathbb{L}$ of the field $\mathbb{Z}_p$.
	
	Then $\alpha^4=-1$. So $\alpha\neq 0$, $\alpha^3=\frac{-1}{\alpha}$ and $\alpha^2+\alpha^{-2}=0$. 
	
	Also I have $\alpha^7=\frac{1}{\alpha}$,   $\alpha^6={\alpha^{-2}}$ and $\alpha^5={-\alpha}={\alpha^{-3}}$.  
	
	Let $y=\alpha+\frac{1}{\alpha}\in\mathbb{L}$. Then $y^2=2$.
	
	As $char(\mathbb{L})=p$, I have $y^p=\alpha^p+\alpha^{-p}$.
	
	Let $p\cong r~(mod~8)$. Then $\alpha^p=\alpha^r$. So $y^p=\alpha^r+\alpha^{-r}$.
	
	If $p\cong\pm 1~(mod~8)$, then $y^p=y$. Therefore $y$ is a solution of $x^p-x=0$ over the field $\mathbb{F}_p$. So $y\in\mathbb{F}_p$.
	
	If $p\cong\pm 3~(mod~8)$, then $y^p=-y$. Therefore $y$ is a not a root of  of $x^p-x$ over the field $\mathbb{Z}_p$. So $y\notin\mathbb{Z}_p$.
\end{proof}
\begin{fact} {\bf\color{blue} (Remove it before publish)}
	Let $\mathbb{K}$ be a number field and let $\mathcal{O}_\mathbb{K}$ be its ring of integers.
	Suppose that $\alpha\in\mathcal{O}_\mathbb{K}$ is a nonzero non-unit such that $\left|N_{\mathbb{K}/\mathbb{Q}}(\alpha)\right|=p,$
	where $p$ is a prime integer. Then $\alpha$ is irreducible in $\mathcal{O}_\mathbb{K}$.
\end{fact}

\begin{proof}
	Suppose, to the contrary, that $\alpha$ is reducible in $\mathcal{O}_\mathbb{K}$.
	Then there exist non-units $\beta,\gamma\in\mathcal{O}_K$ such that $\alpha=\beta\gamma$.
	
	By the multiplicativity of the field norm,
	\[
	\left|N_{\mathbb{K}/\mathbb{Q}}(\alpha)\right|
	=
	\left|N_{\mathbb{K}/\mathbb{Q}}(\beta)\right|
	\left|N_{\mathbb{K}/\mathbb{Q}}(\gamma)\right|.
	\]
	Since $\beta$ and $\gamma$ are non-units in $\mathcal{O}_K$, I have $\left|N_{K/\mathbb{Q}}(\beta)\right|>1$ and $\left|N_{K/\mathbb{Q}}(\gamma)\right|>1$.
	Thus $p
	=
	\left|N_{\mathbb{K}/\mathbb{Q}}(\beta)\right|
	\left|N_{\mathbb{K}/\mathbb{Q}}(\gamma)\right|$
	is a product of two integers greater than $1$, contradicting the
	primality of $p$. Therefore, $\alpha$ is irreducible in
	$\mathcal{O}_\mathbb{K}$.
\end{proof}
\begin{theorem} {\bf\color{blue} (Remove it before publish)}
	Odd prime integer $p$ is prime in $\mathbb{Z}[\sqrt{2}]$, if and only if $p\cong \pm 3~(mod~8)$
\end{theorem}
\begin{proof}
	Third Isomorphism theorem, Ideals $I\subseteq J$ of the ring $R$
	
	Then $(R/I)/(J/I)\cong R/J$
	
	Let $I_1$ and $I_2$ be the ideal generated by $p$ in the ring $\mathbb{Z}[x]$ and $\mathbb{Z}[\sqrt{2}]$ respectively and  $I_3$ and $I_4$ be the ideal generated by $x^2-2$ in the ring $\mathbb{Z}[x]$ and $\mathbb{Z}_p[x]$ respectively.
	
	Case 1 :  Choose $R=\mathbb{Z}[x]$, $J=(x^2-2,p)$, $I_1=(p)$ 
	
	$\implies \mathbb{Z}_p[x]/I_4\cong \mathbb{Z}[x]/(x^2-2,p)$
	
	Case 2 :  Choose $R=\mathbb{Z}[x]$, $J=(x^2-2,p)$, $I_3=(x^2-2)$ 
	
	$\implies \mathbb{Z}[\sqrt{2}]/I_2\cong \mathbb{Z}[x]/(x^2-2,p)$
	
	So one can see that $R/I_1=\mathbb{Z}_p[x]$ and $R/I_3=\mathbb{Z}[\sqrt{2}]$
	
I want to show that $J/I_1\cong I_4$ and $J/I_3\cong I_2$
	
	Any element in $J$ can looks like $p\cdot f(x)+(x^2-2)\cdot g(x)$ where $f(x),g(x)\in\mathbb{Z}[x]$.
	
	By division algorithm, there exist $q(x),r_f(x)$ in $\mathbb{Z}[x]$ such that
	
	$f(x)=(x^2-2)q(x)+r_f(x)$ where either $r_f(x)=0$ or $deg(r_f(x))\leq 1$.
	
	Here $r_f(x)$ is either 0 or $deg(r_f(x))=1\text{ or }0$.
	
	So  $r_f(x)=a_f+b_fx$ for some $a_f,b_f\in\mathbb{Z}$.
	
	Then $f(x)=(x^2-2)q(x)+(a_f+b_fx)$ for some $a_f,b_f\in\mathbb{Z}$.
	
	Let $g(x)=\sum_{k=1}^n c_k~x^k$. 
	
	By division algorithm of integer, there exist integer $r_k$ such that $c_k=p\cdot s_k+r_k$ where $0\leq r_k\leq p-1$. In other words $r_k\in\mathbb{F}_p$.
	
	So $\widetilde{g}(x)=\sum_{k=1}^n r_kx^k$ can be considered as the member of both $\mathbb{Z}[x]$ and $\mathbb{Z}_p[x]$.
	
	Let $h(x)=\sum_{k=1}^n s_kx^k$.Then
	
	$\begin{aligned}&p\cdot f(x)+(x^2-2)\cdot g(x)\\=&p(x^2-2)\cdot [q(x)+h(x)]+p\cdot (a_f+b_fx)+(x^2-2)\cdot \widetilde{g}(x)(x)\end{aligned}$
	
	I know $I_1=p\cdot\mathbb{Z}[x]$,  $I_2=p\cdot\mathbb{Z}[\sqrt{2}]$, $I_3=(x^2-2)\mathbb{Z}[x]$
	
	and  $I_4=(x^2-2)\mathbb{Z}_p[x]$. So
	
	$\begin{aligned}&\text{ any element in }J/I_1\\=&p\cdot f(x)+(x^2-2)\cdot g(x) +I_1\\=&(x^2-2)\cdot \widetilde{g}(x)+p\cdot (a_f+b_fx)+p(x^2-2)\cdot [q(x)+h(x)]+p\cdot\mathbb{Z}[x]\\=&(x^2-2)\cdot \widetilde{g}(x)+p\cdot\mathbb{Z}[x]\\=&(x^2-2)\cdot \widetilde{g}(x)+I_1\end{aligned}$
	
	Here $\widetilde{g}(x)$ is in $\mathbb{Z}_p[x]$. 
	
	So  $\varphi:J\to I_4$ defined by $\varphi(p\cdot f(x)+(x^2-2)\cdot g(x))=(x^2-2)\cdot \widetilde{g}(x)$ is surjective ring homomorphism. Here $Ker(\varphi)=I_1$.
	
	By 1st Isomorphism theorem, $J/I_1\cong I_4$
	
	$\begin{aligned}&\text{ any element in }J/I_3\\=&p\cdot f(x)+(x^2-2)\cdot g(x) +I_3\\=&p\cdot (a_f+b_fx)+(x^2-2)\cdot \widetilde{g}(x)+p(x^2-2)\cdot [q(x)+h(x)]+(x^2-2)\cdot\mathbb{Z}[x]\\=&p\cdot (a_f+b_fx)+(x^2-2)\cdot\mathbb{Z}[x]\\=&p\cdot (a_f+b_fx)+I_3\end{aligned}$
	
	So  $\psi:J\to I_2$ defined by 
	
	$\psi(p\cdot f(x)+(x^2-2)\cdot g(x))=p\cdot (a_f+b_f\sqrt{2})=p\cdot f(\sqrt{2})$ is surjective ring homomorphism. Here $Ker(\psi)=I_3$.
	
	By 1st Isomorphism theorem, $J/I_3\cong I_2$

	Let $I_2=(p)$ be ideal generated by $p$ in $\mathbb{Z}[\sqrt{2}]$.
	
	Let $I_4=(x^2-2)$ be ideal generated by $x^2-2$ in $\mathbb{Z}_p[x]$. 
	
	Then $``p$ is irreducible$"$ if and only if
	$``\mathbb{Z}[\sqrt{2}]/(p)\cong \mathbb{Z}[x]/(x^2-2,p)\cong \mathbb{Z}_p[x]/(x^2-2)$ is field $"$ if and only if $``$ no element $y$ in $\mathbb{Z}_p$ so that $y^2=2"$.
	
	So $p\ncong\pm 1~(mod~8)$, i.e., $p\cong \pm 3~(mod~8)$
\end{proof}
\begin{lemma}[Thue's lemma {\cite[theorem 2.33]{Shoup2009}}] {\bf\color{blue} (Remove it before publish)} Let $a,n,r,s\in\mathbb{Z}$, such that $0<r\leq n<rs$, there are two integers $u$ and $v$ such that $u\cong av~(mod~n)$ and $0\leq|u|<r,~0\leq|v|<s$. but both $u,v$ not simultaneously 0. 
\end{lemma}
\begin{theorem} {\bf\color{blue} (Remove it before publish)}
	Let $p$ be an odd prime such that $p\equiv \pm1\pmod 8.$
	Then $p$ is reducible in $\mathbb{Z}[\sqrt{2}]$. More precisely, $p$ can be written as a product of two non-units in $\mathbb{Z}[\sqrt{2}]$.
\end{theorem}

\begin{proof}
	Since $p\equiv \pm1\pmod 8,$ I have Legendre symbol $\left(\frac{2}{p}\right)=1.$
	Hence there exists $y\in\mathbb{Z}_p$ such that $y^2\equiv 2\pmod p.$
	
	I use Thue's lemma(See \cite[Theorem~2.33]{Shoup2009}) in the following form: if $a,n,r,s\in\mathbb{Z}$ satisfy $0<r\leq n<rs,$ then there exist integers $u,v$, not both zero, such that $u\equiv av\pmod n$, $|u|<r$, $|v|<s$.

	Set $r=s=\lfloor\sqrt{p}\rfloor+1,
	\quad
	a=y,
	\quad
	n=p.$ Since $\sqrt{p}<\lfloor\sqrt{p}\rfloor+1,$ I have $
	p<\bigl(\lfloor\sqrt{p}\rfloor+1\bigr)^2=rs.$
	Thus Thue's lemma yields integers $u,v$, not both zero, such that
	\[
	u\equiv yv\pmod p,
	\qquad
	|u|<r,
	\qquad
	|v|<s.
	\]
	Since $u,v$ are integers, this gives $|u|,|v|\leq\lfloor\sqrt p\rfloor,$
	and hence $0\leq u^2,v^2\leq p.$
	
	Moreover, $u^2\equiv y^2v^2\equiv2v^2\pmod p.$
	
	Therefore $u^2-2v^2=pt$ for some $t\in\mathbb{Z}$. Since $0\leq u^2\leq p$ and $0\leq2v^2\leq2p,$ I obtain $-2p\leq u^2-2v^2\leq p.$
	
	Consequently, $u^2-2v^2\in\{-2p,-p,0,p\}.$
	
	The case $u^2-2v^2=0$ is impossible. Indeed, $u^2=2v^2$ implies $u=v=0$, since $\sqrt2\notin\mathbb{Q}$, contradicting the fact that $u$ and $v$ are not both zero.
	
	I now consider the remaining three cases.
	
	\medskip
	
	\noindent
	\textbf{Case 1:} Suppose $u^2-2v^2=-2p$.
	Then $u^2$ is even, so $u=2u_1$ for some $u_1\in\mathbb{Z}$. Hence $4u_1^2-2v^2=-2p,$ and therefore $p=v^2-2u_1^2$. 
	Thus
	\[
	p=(v+u_1\sqrt2)(v-u_1\sqrt2).
	\]
	Here $N(v+u_1\sqrt2)=p$, prime implies $v+u_1\sqrt2$ is irreducible.
	\medskip
	
	\noindent
	\textbf{Case 2:} Suppose $u^2-2v^2=-p$. Then
	\[
	p=2v^2-u^2
	=(u+v\sqrt2)(-u+v\sqrt2).
	\]
	Here $N(u+v\sqrt2)=-p$, prime implies $u+v\sqrt2$ is irreducible.
	\medskip
	
	\noindent
	\textbf{Case 3:} Suppose $u^2-2v^2=p$.
	Then
	\[
	p=(u+v\sqrt2)(u-v\sqrt2).
	\]
	Here $N(u+v\sqrt2)=p$, prime implies $u+v\sqrt2$ is irreducible.
	
	In each case, the two factors belong to $\mathbb{Z}[\sqrt2]$. Moreover, their norms have absolute value $p$ (or, in the first case, the displayed factors have norm $\pm p$). Since $p>1$, neither factor is a unit, because the units of $\mathbb{Z}[\sqrt2]$ have norm $\pm1$.
	
	Therefore $p$ is a product of two non-units in $\mathbb{Z}[\sqrt2]$. Hence $p$ is reducible in $\mathbb{Z}[\sqrt2]$.
\end{proof}


\end{document}